\documentclass[11pt,reqno]{amsart}
\usepackage[foot]{amsaddr}

\allowdisplaybreaks

\usepackage{parskip}
\usepackage{fullpage}
\usepackage{amssymb,amsfonts,amsmath,amsthm}
\usepackage{verbatim}

\usepackage{hyperref}
\usepackage[dvipsnames]{xcolor}
\newcommand\myshade{85}
\colorlet{mylinkcolor}{violet}
\colorlet{mycitecolor}{YellowOrange}
\colorlet{myurlcolor}{Aquamarine}
\hypersetup{
	linkcolor  = mylinkcolor!\myshade!black,
	citecolor  = mycitecolor!\myshade!black,
	urlcolor   = myurlcolor!\myshade!black,
	colorlinks = true,
}

\usepackage[shortlabels]{enumitem}
\newlist{sumside}{itemize}{2}
\setlist[sumside]{topsep=0pt,label=-,leftmargin=2em,noitemsep}

\usepackage{etoolbox}
\patchcmd{\section}{\scshape}{\bfseries}{}{}
\makeatletter
\renewcommand{\@secnumfont}{\bfseries}
\makeatother

\makeatletter
\patchcmd{\@settitle}{\uppercasenonmath\@title}{\large}{}{}
\patchcmd{\@setauthors}{\MakeUppercase}{}{}{}
\makeatother

\usepackage[compress]{cite}

\usepackage{newtxtext}
\usepackage[libertine]{newtxmath}

\usepackage{ stmaryrd}
\newtheorem{thm}{Theorem}
\newtheorem{define}[thm]{Definition}

\newtheorem{cor}[thm]{Corollary}

\newtheorem*{exam-non}{Example}

\makeatletter
\def\th@definition{%
	\thm@headfont{\bfseries}
	\normalfont 
}
\makeatother

\newcommand{\ignore}[1]{}
\newcommand{\nMod}[2]{
	\not\equiv #1\,\,(\text{mod}\,\,#2)
	}
\newcommand{\Mod}[2]{
  \,\equiv\, #1\,\,(\text{mod}\,\,#2)
}

\newcommand{\NN}{\mathbb{N}}

\newcounter{idfamily}

\newcounter{familyexample}[idfamily]

\usepackage{xstring,refcount}

\allowdisplaybreaks

\begin{document}

\title{
On a dual and an overpartition generalization of a family of identities of Andrews
}
\author{Shashank Kanade}
\address{University of Alberta, Edmonton, Canada}
\thanks{S.K. is currently supported by PIMS Post-doctoral Fellowship awarded by Pacific Institute for the Mathematical Sciences}
\email{kanade@ualberta.ca}

\author{Matthew C.\ Russell}
\address{Rutgers, The State University of New Jersey, Piscataway, USA}
\email{russell2@math.rutgers.edu}

\begin{abstract}
We present a dual of a family of partition identities of Andrews involving partitions with no repeated odd parts (among other conditions), along with an overpartition generalization that encapsulates both families. These were discovered during the course of research for an upcoming article by the authors along with Debajyoti Nandi. The proof uses Appell's comparison theorem.
\end{abstract}

\maketitle

\section{Introduction}

Recall the well-known partition theorem of I.\ Schur, (see  \cite{schur} and \cite[Ch. 7]{And-book}):
\begin{thm}[Schur]
Partitions of a positive integer $n$ where each part is $\Mod{\pm 1}{6}$ 
are equinumerous with partitions of $n$ where the difference between
adjacent parts is at least 3, and at least 6 if these parts are divisible by $3$. 
\end{thm}	
During the course of a proof of this theorem using Appell's comparison theorem
(recalled below, Theorem \ref{thm:Appell}), roughly 50 years ago G.\ E.\ Andrews was led
to the following family of partition identities.
\begin{thm}[Thm.\ 3 \cite{And-SecondSchur}]\label{thm:And1}
For $k\geq 2$, let $B_{k-1,k}(n)$ be the number of partitions of a non-negative integer $n$ in which each part is either even but $\nMod{4k-2}{4k}$ or odd and $\Mod{2k-1,4k-1}{4k}$. Also, let $C_{k-1,k}(n)$ be the number of partitions of $n$ in which if an odd part $2j+1$ is present, then none of the other parts are equal to any of $2j+1,2j,\dots,2j-2k+3$, and the smallest part is not equal to any of $1,3,\dots,2k-3$. Then, $B_{k-1,k}(n) = C_{k-1,k}(n)$ for all $n$.
\end{thm}
The meaning of the subscripts $\{k-1,k\}$ (in the previous theorem) and $\{0,k\}$ (in the following theorem) will be made clear later in this paper.
Incidentally, this same family also arose during the course of our recent search  (jointly with D.\ Nandi) for identities of Rogers-Ramanujan-MacMahon type \cite{KNR}. In this note, we provide a dual version of this family:
\begin{thm}\label{thm:dual}
For $k\geq 2$, let $B_{0,k}(n)$ be the number of partitions of a non-negative integer $n$ in which each part is either even but $\nMod{2}{4k}$ or odd and $\Mod{1,2k+1}{4k}$. Also, let $C_{0,k}(n)$ be the number of partitions of $n$ in which if an odd part $2j+1$ is present, then none of the other parts are equal to any of $2j+1,2j+2,\dots,2j+2k-1$. Then, $B_{0,k}(n) = C_{0,k}(n)$ for all $n$.
\end{thm}

As an example of this family, we note that setting $k=2$ results in $B_{0,2}(n)$ counting partitions with all parts congruent to $\Mod{0,1,4,5,6}{8}$, while $C_{0,2}(n)$ counts partitions where the occurrence of an odd part forbids the occurrence of a part  $0, 1$ or $2$ greater. More specifically, for $n=10$, the $B_{0,2}(10)=10$ partitions counted by the ``product side'' are $9+1$, $8+1+1$, $6+4$, $6+1+1+1+1$, $5+5$, $5+4+1$, $5+1+1+1+1+1$, $4+4+1+1$, $4+1+1+1+1+1+1$, and $1+1+1+1+1+1+1+1+1+1$, while the $C_{0,1}(10)=10$ partitions counted by the ``sum side'' are $10$, $9+1$, $8+2$, $7+3$, $6+4$, $6+2+2$, $5+4+1$, $4+4+2$, $4+2+2+2$, and $2+2+2+2+2$. 

Over the past decade or so, there has been a lot of interest in exploring overpartition analogues of classical partition identities, 
(as a small and by no means exhaustive sample, see papers by Chen et al., Corteel, Lovejoy, and Dousse \cite{CoLo-over,Lo-RRover,CSS-brover,Dousse-Schurover}).
Overpartitions are partitions in which last occurrence of any part may appear overlined. 

The fact that odd parts are not allowed to be repeated (though even parts may be repeated arbitrarily many times) in the above theorems suggests that both are actually special cases of an overpartition theorem. We now present an overpartition generalization that can be used to recover Theorems \ref{thm:And1} and \ref{thm:dual} upon appropriate specializations.

\begin{thm}\label{thm:mainoverthm}
For $k \ge 2$,  let $D_k(m,n)$ be the number of overpartitions of $n$ with exactly $m$ overlined parts, subject to the following conditions:
\begin{itemize}
    \item If an overlined part $\overline{b}$ appears, then 
	all of the non-overlined parts $b, b+1,\dots, b+k-2$ are forbidden to appear.
	\item If an overlined part $\overline{b}$ appears, then 
	all of the overlined parts $\overline{b+1},\overline{b+2},\dots, \overline{b+k-1}$ are forbidden to appear.
\end{itemize}
Then, 
\[\sum_{m,n \ge 0} D_k(m,n)a^m q^n = \dfrac{\left(-aq;q^k\right)_{\infty}}{(q;q)_{\infty}}.\]
\end{thm}
Here and throughout, we have used the standard $q$-Pochhammer notation 
$\left(b;q\right)_t =  \prod\limits_{0\leq m < t}(1-bq^m)$. Specializing $a=1$ means that the product side is simply overpartitions where only parts $\Mod{1}{k}$ can be overlined.

We prove this family again using Appell's comparison theorem.
Appell's theorem is an important tool in the theory of $q$-series and
has been successfully exploited in many cases, see for example,
\cite{Dousse-Schurover,Dousse-And,Dousse-And2}. After we prove this theorem, we will derive corollaries for all $i \in \left\{0,\dots,k-1\right\}$, where the cases corresponding to $i=0$ and $i=k-1$ are Theorem \ref{thm:dual} and Theorem \ref{thm:And1}, respectively.

\section{Proofs}

We begin by recalling Appell's comparison theorem.

\begin{thm}[Appell's comparison theorem {\protect{\cite[page 101]{D-taylor}}}] \label{thm:Appell}
	Let $p_n$ be a sequence of positive real numbers such that
	the series $\sum_{n\geq 0}{p_n}$ diverges and such that
	the radius of convergence of $\sum_{n\geq 0}p_nz^n$ is $1$.
	Let $a_n$ be another sequence such that $\lim_{n\rightarrow\infty}{a_n/p_n}=s$.
	Then, the radius of convergence of $\sum_{n\geq 0}a_nz^n$ is also 
	$1$ and moreover,
	$\lim_{r\rightarrow 1}(\sum_{n\geq 0}a_nr^n)/(\sum_{n\geq 0}p_nr^n)=s$.
\end{thm}

Let us provide a special case ($p_n\equiv 1$) of Appell's theorem that will prove to be more amenable to our calculations.
First, let us define the notion of limit of a sequence of formal series that we shall employ.

\begin{define}\label{defn:limit}
	Let \[f_n(q)=\sum_{j\geq 0}a^{(n)}_jq^j,\quad n\in\NN\] be a sequence of elements of $R[[q]]$
	where $R$ is any ring.
	We say that $\lim\limits_{n\rightarrow\infty}f_n(q)$ exists if
	for all $M\in \NN$ there exists an $N\in\NN$ with
	$q^{M+1}\big\vert \left(f_{n_1}(q) - f_{n_2}(q)\right)$ for all $n_1,n_2\geq N$; in other words, the coefficient of any 
	arbitrary power of $q$ in $f_n(q)$ stabilizes as $n\rightarrow \infty$.
	(In practice, this condition is usually very easy to verify.)
	If the limit exists, we let it be
	\begin{align*}
	\lim\limits_{n\rightarrow\infty}f_n(q)&= \sum_{j\geq 0} \left(\lim\limits_{t\rightarrow\infty}a^{(t)}_j\right)q^j.
	\end{align*}
\end{define}

\begin{cor}\label{cor:formalAppell}
	Let \[f_n(q)=\sum_{j\geq 0}a^{(n)}_jq^j,\quad n\in \NN\] be a sequence of elements of $R[[q]]$
	where $R$ is some ring
	such that $f_{\infty}(q):=\lim\limits_{n\rightarrow\infty}f_n(q)$ exists.
	Let 
	\begin{align*}
	F(x,q) &= \sum_{n\geq 0}f_n(q)x^n.
	\end{align*}
	Then,
	\[\lim\limits_{x\rightarrow 1}\left((1-x)F(x,q)\right)=f_\infty(q).\]
\end{cor}

We now prove Theorem \ref{thm:mainoverthm}. 

\begin{proof}[Proof of Theorem \ref{thm:mainoverthm}]
Fix $k \ge 2$. Let $p_j(m, n)$ be the number of overpartitions of $n$ 
with $m$ overlined parts
that satisfy the conditions in Theorem \ref{thm:mainoverthm}, with the further restriction that all parts are $\leq j$. Let $r_j(m,n)$ be the number of overpartitions of $n$ counted by $p_j(m,n)$ where $\overline{j}$, $\overline{j-1}$, \dots, $\overline{j-k+2}$ do not appear (that is, the largest possible overlined part is $\overline{j-k+1}$). Then, let 
\begin{align*}
P_j(a,q) &= \sum_{m,n \ge 0} p_j(m,n)a^mq^n,\\
R_j(a,q) &= \sum_{m,n \ge 0} r_j(m,n)a^mq^n,
\end{align*}
we let $P_0=R_0=1$.
It is clear that
\[R_{\infty}(a,q) = P_{\infty}(a,q) = \sum_{m,n \ge 0} D_k(m,n)a^m q^n.\]

Let 
\[F(a,x,q) = \sum_{j\geq 0}R_j(a,q)x^j.\]
Observe that the following recursion and initial conditions are satisfied:
\begin{align*}
R_j(a,q) &= \dfrac{1}{1-q^j}R_{j-1}(a,q) + \dfrac{aq^{j-k+1}}{1-q^j} R_{j-k}(a,q),\,\, j\geq k.\\
R_j(a,q) &= \dfrac{1}{(q;q)_j},\,\, 0\leq j < k.
\end{align*}
Note the following alternate way to write the recursion and the initial conditions:
\begin{align*}
R_j(a,q) &= \dfrac{1}{1-q^j}R_{j-1}(a,q) + \dfrac{aq^{j-k+1}}{1-q^j} R_{j-k}(a,q),\,\, j\geq 1.\\
R_0(a,q) &= 1,\quad R_{j}(a,q)=0\,\,\text{for } -k<j<0,
\end{align*}
which immediately gets us to
\[
(1-x)F(a,x,q)=F(a,xq,q)+ax^kqF(a,xq,q) = (1+ax^kq)F(a,xq,q).
\]
Noting that 
\[
\lim\limits_{n\rightarrow\infty}F(a,xq^n,q)=R_0(a,q)=1,
\]
we obtain
\[F(a,x,q)= \prod\limits_{j\geq 0}\dfrac{1+ax^kq^{jk+1}}{1-xq^j}.\]
Finally, by Corollary \ref{cor:formalAppell},
\[
R_{\infty}(a,q)=
\lim\limits_{x\rightarrow 1}((1-x)F(a,x,q))
= \lim\limits_{x\rightarrow 1}
\left(\prod\limits_{j\geq 0}\dfrac{1+ax^kq^{jk+1}}{1-xq^{j+1}}\right)
=
\dfrac{\left(-aq;q^k\right)_{\infty}}{(q;q)_{\infty}}.
\]
\end{proof}

Now, Theorems \ref{thm:And1} and \ref{thm:dual} can be recovered by appropriate specializations. Letting $(a,q)\mapsto \left(q^{-1},q^2\right)$ (that is, we map every nonoverlined part $j\mapsto 2j$ and every overlined part $\overline{j}\mapsto 2j-1$) gives us Theorem \ref{thm:dual}, while using $(a,q)\mapsto \left(q^{2k-3},q^2\right)$ (now mapping $j\mapsto 2j$ and every overlined part $\overline{j}\mapsto 2j+2k-3$) provides us with Theorem \ref{thm:And1}.

However, many more corollaries can be found. For $k\ge 2$, by choosing $i \in \left\{0,\dots,k-1\right\}$ and letting $(a,q)\mapsto \left(q^{2i-1},q^2\right)$, we obtain:
\begin{cor}
	Let $B_{i,k}(n)$ be the number of partitions of a non-negative integer $n$ in which each part is either even but $\nMod{4i+2}{4k}$ or odd and $\Mod{2i+1,2k+2i+1}{4k}$. Also, let $C_{i,k}(n)$ be the number of partitions of $n$ in which if an odd part $2j+1$ is present, 
	then none of the other even parts are equal to any of  
	$2j-2i+2, 2j-2i+4, \dots, 2j+2k-2i-2$,
	none of the other odd parts are equal to any of
	$2j+1, 2j+3, \dots, 2j+2k-1$
	and the smallest odd part is at least $2i+1$. 
	Then, $B_{i,k}(n) = C_{i,k}(n)$ for all $n$.
\end{cor}

We leave it to the reader to work out identities related to the specializations
$q\mapsto q^t$ for $t>2$.


\begin{thebibliography}{99}

\bibitem{And-SecondSchur} G. E. Andrews. On Schur's second partition theorem. {\it Glasgow Math. J.}, 8:127--132, 1967.

\bibitem{And-book} G. E. Andrews, {\it The theory of partitions.} Cambridge Mathematical Library. Cambridge University Press, Cambridge, 1998. Reprint of the 1976 original.

\bibitem{CSS-brover} W. Y.-C. Chen, D. D.-M. Sang, and D. Y.-H. Shi. An overpartition analogue of Bressoud's theorem of Rogers-Ramanujan-Gordon type. {\it Ramanujan J.}, 36(1-2):69--80, 2015.

\bibitem{CoLo-over} S. Corteel and J. Lovejoy. Overpartitions. {\it Trans. Amer. Math. Soc.}, 356(4):1623--1635, 2004.

\bibitem{D-taylor} P. Dienes. {\it The Taylor series: an introduction to the theory of functions of a complex variable.} Dover Publications, Inc., New York, 1957.

\bibitem{Dousse-Schurover} J. Dousse. On generalizations of partition theorems of Schur and Andrews to overpartitions. {\it Ramanujan J.}, 35(3):339--360, 2014.

\bibitem{Dousse-And} J. Dousse. A generalisation of a partition theorem of Andrews. {\it Monatsh. Math.},
179(2):227--251, 2016.

\bibitem{Dousse-And2} J. Dousse. A generalisation of a second partition theorem of Andrews to overpartitions. {\it J. Combin. Theory Ser. A}, 145:101--128, 2017.

\bibitem{KNR} S. Kanade, Shashank, D. Nandi, M. C. Russell. A variant of {\tt {I}dentity{F}inder} and some new identities of {Rogers-Ramanujan-MacMahon} type. In preparation.

\bibitem{Lo-RRover} J. Lovejoy. Overpartition theorems of the Rogers-Ramanujan type. {\it J. London Math. Soc. (2)}, 69(3):562--574, 2004.

\bibitem{schur} I. Schur. Zur additiven Zahlentheorie.
{\it S.-B. Akad. Wiss. Berlin}, pages 488--495, 1926.


\end{thebibliography}
\end{document}